\documentclass[12pt,a4paper]{amsart}         
\usepackage[margin=2cm]{geometry}
\usepackage[backend=bibtex,url=false,doi=true,isbn=false,giveninits=true,style=numeric,sorting=nyt]{biblatex}
\usepackage{amsmath,amsfonts,amssymb,amsthm}
\usepackage{mathrsfs,epsfig,epstopdf,url}
\usepackage{enumerate}
\usepackage[table]{xcolor}
\usepackage{multirow}

\theoremstyle{definition}
\newtheorem{definition}{Definition}%

\theoremstyle{plain}
\newtheorem{theorem}[definition]{Theorem}
\newtheorem{lemma}[definition]{Lemma}
\newtheorem{proposition}[definition]{Proposition}
\newtheorem*{problem}{Problem}

\newtheorem{remark}[definition]{Remark}

\newcommand{\E}{\mathsf{E}}
\newcommand{\D}{\mathsf{D}}

\newcommand{\bm}[1]{\boldsymbol{#1}}

\newcommand{\K}{\mathcal{K}}
\DeclareMathOperator{\Aff}{Aff}
\newcommand{\NL}{\mathrm{NL}_{\bm{v}}}

\title{Thin Sidon sets and the nonlinearity of vectorial Boolean functions}

\author{G\'abor P. Nagy}
\address{Department of Algebra \\
	Budapest University of Technology and Economics\\
	M\H{u}egyetem rkp 3\\
	H-1111 Budapest, Hungary}
\address{Bolyai Institute \\
	University of Szeged \\
	Aradi v\'ertan\'uk tere 1\\
	H-6720 Szeged, Hungary}
\email{nagyg@math.u-szeged.hu}

\thanks{Support provided from the National Research, Development and Innovation Fund of Hungary, NKFIH-OTKA Grant SNN 132625, and the Program of Excellence TKP2021-NVA-02 ate the Budapest University of Techonology and Economics.}

\keywords{Sidon sets, vectorial Boolean functions, APN functions, Hamming distance, error correcting codes}
\subjclass[2010]{94A60,94D10,05B25}

\begin{document}

\begin{abstract}
The vectorial nonlinearity of a vector valued function is its distance from the set of affine functions. In 2017, Liu, Mesnager and Chen conjectured a general upper bound for the vectorial linearity. Recently, Carlet proved a lower bound in terms of the differential uniformity. In this paper, we improve Carlet's lower bound. Our method is elementary, it relies on the fact that the level sets of an APN functions are Sidon sets. We give a survey on Sidon sets in elementary abelian 2-groups. We study the completeness problem of Sidon sets obtained from hyperbolas and ellipses of the finite affine plane. 
\end{abstract}

\maketitle

\section{Introduction}

Vectorial Boolean functions $\mathbb{F}_2^n\to \mathbb{F}_2^m$, also called \textit{substitution boxes,} play a central role in symmetric key block ciphers. By being the only nonlinear components of the ciphers, they provide \textit{confusion.} The study of the nonlinear properties of vectorial Boolean functions is fundamental for the evaluation of the resistance of the block cipher against the main attacks, such as the differential attack and the linear attack, see \cite{Carlet2021,Carlet2021b} and their references. The main metrics of these nonlinear properties are the \textit{differential uniformity} (the lower is the less linear), the \textit{nonlinearity,} and the \textit{vectorial nonlinearity.} 

Let $V,V'$ be vector spaces over $\mathbb{F}_2$. Let $f,g:V\to V'$ be functions.
\begin{enumerate}[(i)]
\item The \textit{Hamming distance} of $f,g$ is
\[d_H(f,g) = |\{x\in V \mid f(x)\neq g(x)\}|.\]
\item $f$ is \textit{affine,} if $f(x+y+z)=f(x)+f(y)+f(z)$ for all $x,y,z\in V$. The set of affine $V\to V'$ maps is denoted by $\Aff(V,V')$.  
\item Let $\omega:V'\to \mathbb{F}_2$ be a nonzero linear functional. The Boolean function $\omega f:V\to \mathbb{F}_2$, $(\omega f)(x)=\omega(f(x))$ is called a \textit{component Boolean function} of $f$. The \textit{nonlinearity} of $f$ is the distance between its component Boolean functions and affine Boolean functions
\[\mathrm{NL}_1(f)=\min_{\substack{\omega\in (V')^*\setminus\{0\} \\ \alpha\in\Aff(V,V')}} d_H(\omega f,\alpha).\]
\item The \textit{vectorial nonlinearity} of $f$ is its distance from the set of affine functions
\[\NL(f) = d_H(f,\Aff(V,V'))= \min_{\alpha\in\Aff(V,V')} d_H(f,\alpha).\]
\item The \textit{differential uniformity} of $f$ is 
\[\delta_f = \max_{\substack{a\in V\setminus\{0\} \\ b\in V'}} |\{x\in V \mid f(x)+f(x+a)=b\}|.\]
\item If $V=V'$ and $\delta_f=2$, then the function $f$ is called \textit{almost perfect nonlinear (APN).} 
\end{enumerate}
An affine map has maximum differential uniformity $\delta_f=|V|$, functions with low differential uniformity are considered to be far from being linear. One has $\delta_f\geq 2$, and by definition, APN functions are those with least possible differential uniformity. The \textit{Walsh-Hadamard transform} provides an effective tool for computation with the vectorial nonlinearity $\mathrm{NL}_1(f)$. However, the computation of the nonlinearity $\NL(f)$ is in general hard. Also, it is difficult to give nontrivial lower and upper bounds for $\NL(f)$. The two nonlinearity measures are linked by the inequality
\[\NL(f) \geq \mathrm{NL}_1(f),\]
that holds for all $f:V\to V'$, see \cite{Liu2017}.

\begin{problem}[Liu-Mesnager-Chen Conjecture \cite{Liu2017}]
If $n=\dim(V)$, $m=\dim(V')$, then for any map $f:V\to V'$, 
\begin{align} \label{eq:LMCconj}
\NL(f)\leq (1-2^{-m})(2^n-2^{n/2})
\end{align}
holds. In particular, if $n=m$, then
\[\NL(f)\leq 2^n-2^{n/2}-1.\]
\end{problem}
One expects that APN functions have high vectorial nonlinearity, near to the upper bound given in \eqref{eq:LMCconj}. Recently, Carlet \cite[Proposition 4]{Carlet2021b} proved a general lower bound for the vectorial nonlinearity, in terms of the differential uniformity:
\begin{align} \label{eq:carlet_df}
\NL(f)\geq 2^n - \sqrt{2^n+\delta_f(2^n-1)}.
\end{align}
For APN functions, this gives
\begin{align} \label{eq:carlet_apn}
\NL(f)> 2^n - \sqrt{3} \cdot 2^{n/2}.
\end{align}
In this paper, we improve Carlet's bound \eqref{eq:carlet_df}. We remark that for the bound \eqref{eq:carlet_apn} on APN functions, the same improvement has been obtained independently by Carlet \cite{Carlet2022b}, as well.

\begin{theorem} \label{thm:nonlinearity}
For all $f:V\to V'$, we have
\[\NL(f)\geq 2^n - \sqrt{\delta_f}\cdot 2^{n/2}-\frac{1}{2}.\]
In particular, for an APN function $f:\mathbb{F}_2^n\to \mathbb{F}_2^n$, 
\[\NL(f)\geq 2^n - \sqrt{2} \cdot 2^{n/2}-\frac{1}{2}.\]
\end{theorem}

Our key observation is the connection between Sidon sets and the differential uniformity. Sidon sets are studied since the 1930's, see \cite{Babai1985,Balogh2021,Erdoes1941,Lindstroem1969,Redman2021} for the combinatorial setting, and \cite{Carlet2022,Carlet2021a} for their importance in cryptography. A set $S$ of integers is a \textit{Sidon set} or a \textit{Sidon sequence,} if $a+b=c+d$, $a,b,c,d\in S$ imply $\{a,b\}=\{c,d\}$. The definition can be generalized to any abelian group $(A,+)$, where special attention has to be paid for involutions, see \cite{Babai1985}. We say that $S\subseteq A$ is a \textit{Sidon set} in $A$, if for any $x,y,z,w\in S$ of which at least three are different,
\[x+y\neq z+w.\]
The Sidon set $S$ is \textit{complete,} if for any $a\in A\setminus S$, $S\cup \{a\}$ is not a Sidon set. In \cite{Carlet2022a,Redman2021} (using the terminology \textit{maximal} instead of complete), the completeness of those Sidon sets has been investigated, that can be obtained as graphs of APN functions. Carlet \cite{Carlet2022a} observed that the incompleteness of the graph of the APN function $f$ is equivalent with the existence of another APN function $g$ such that their Hamming distance is $d_H(f,g)=1$. In 2018, Budaghyan, Carlet, Helleseth, Li and Sun conjectured that $d_H(f,g)=1$ is impossible for two APN functions $f,g$, see \cite[Conjecture 2]{Budaghyan2018}. 

For any Sidon set $S$ of $A$, the upper bound $|S|< \sqrt{2|A|}+\frac{1}{2}$ is immediate. If $A$ has odd order, then we even have $|S|<\sqrt{|A|}+\frac{1}{2}$. If $A$ is cyclic, then a result of Erd\H{o}s and Tur\'an \cite{Erdoes1941} implies $|S|\leq (1+o(1))\sqrt{|A|}$. For many classes of abelian groups, including cyclic groups and groups of odd prime exponent, Sidon sets of size $(1+o(1))\sqrt{|A|}$ are known to exist. Hence, in these groups, the upper and lower bounds for $|S|$ are very close. However, for elementary abelian $2$-groups, up to the author's knowledge, there is no upper bound which is significantly better than $\sqrt{2|A|}$. At the same time, the best known constructions satisfy $|S|\leq \sqrt{|A|}+O(1)$. To be more precise, fix $q=2^m$ and let $A$ be the additive group of $\mathbb{F}_q^2$. The classical example of a Sidon set in $A$ is due to Lindström \cite{Lindstroem1969}:
\begin{align} \label{eq:cubeexample}
\{(x,x^3) \mid x\in \mathbb{F}_q\}.
\end{align}
This has size $q=\sqrt{|A|}$. It is much harder to find Sidon sets of order $q+1$. Using the connection between double-error correcting linear codes and Sidon sets, examples of size $q+1$ were obtained from non-binary BCH codes and binary Goppa codes. Recently, Carlet and Mesnager \cite{Carlet2022} showed that cyclic subgroups of order $q+1$ are Sidon sets in $\mathbb{F}_{q^2}$. Again, up to the author's knowledge, only sporadic examples of Sidon sets are known, where $|S|\geq q+2$. In this paper we present an infinite class of examples with $|S|=q+2$, and show their completeness. 

We construct these sets as point sets of the affine plane over the finite field $\mathbb{F}_q$. The points of $AG(2,q)$ are elements of $\mathbb{F}_q^2$, lines are given by homogeneous linear equations $aX+bY+c=0$, and \textit{conics} are given by irreducible polynomials of degree $2$ in two variables. In characteristic $2$, the tangents of a conic pass through a common point $N$, which is called the \textit{nucleus} of the conic. 

\begin{theorem} \label{thm:conics}
Let $m\geq 4$ be an integer, $q=2^m$. Let $C$ be an ellipse or a hyperbola in the affine plane $AG(2,q)$. Let $N$ be the nucleus of $C$. 
\begin{enumerate}[(i)]
\item When $m$ is even and $C$ is a hyperbola, or when $m$ is odd and $C$ is an ellipse, then $C$ is a complete Sidon set in $\mathbb{F}_q^2$. 
\item When $m$ is odd and $C$ is a hyperbola, or when $m$ is even and $C$ is an ellipse, then $C\cup \{N\}$ is a complete Sidon set in $\mathbb{F}_q^2$. 
\end{enumerate}
\end{theorem}

As an ellipse of $AG(2,q)$ has $q+1$ points, we obtain Sidon sets of size $\sqrt{|A|}+2$ for $m$ even. We will also show that the cyclic subgroup and the binary Goppa code constructions are isomorphic to an ellipse in $AG(2,q)$, see Remark \ref{rem:ellipse}. 

\begin{problem}[Sidon sets in elementary abelian $2$-groups]
For which integers $n$ is true that if $A$ is an elementary abelian group of order $2^n$, and $S$ is a Sidon set in $A$, then $|S|\leq \sqrt{|A|}+2$? 
\end{problem}

We are only aware of a single value of $n$ which does not satisfy this property: $n=11$. By extending binary Goppa codes, Chen \cite{Chen1991} constructed a binary linear code with parameters $[47,36,5]$. This gives rise to a Sidon set $S$ of size $48$ in $A=\mathbb{F}_2^{11}$. Hence, $|S|>\sqrt{|A|}+2\approx 47.25$. 

The paper is structured as follows. In section \ref{sec:sidon-prop}, we enumerate the most important properties of Sidon sets in abelian groups in general, and in elementary abelian 2-group in particular. In section \ref{sec:sidon-codes}, we present the relation between Sidon sets of $\mathbb{F}_2^n$, and binary linear double-error correcting codes. The known constructions of such codes enable us to construct Sidon sets of size $2^{n/2}+O(1)$ for $n$ even, and to determine the maximum size of a Sidon set for $n\leq 10$. In section \ref{sec:nonlin}, we state the relation between Sidon sets and the level sets of APN functions. We use this link to give an elementary proof for Theorem \ref{thm:nonlinearity}. In sections \ref{sec:conics} and \ref{sec:completeness}, we construct Sidon sets using ellipses and hyperbolas of the affine plane $AG(2,q)$, and we prove their completeness. We compile these results to prove Theorem \ref{thm:conics} in section \ref{sec:proof}.

\section{Properties of Sidon sets} \label{sec:sidon-prop}

We start by repeating the definition of a Sidon set, and add the definition of a closely related concept.

\begin{definition}
Let $A$ be a finite abelian group.
\begin{enumerate}[(i)]
\item We say that $S\subseteq A$ is a \textit{Sidon set} in $A$, if for any $x,y,z,w\in S$ of which at least three are different,
\[x+y\neq z+w.\]
Equivalently, $x-z\neq w-y$. 
\item Let $t$ be a positive integer. The subset $T$ of $A$ is \textit{$t$-thin Sidon,} if for any $a\in A$, $|T\cap (T+a)|\leq t$. 
\end{enumerate}
\end{definition}

The following lemmas are well-known, the proofs are very easy. In order to be self-contained, we give a short hint. 

\begin{lemma}
Let $T$ be a $t$-thin Sidon set in the abelian group $A$.
\begin{enumerate}[(i)]
\item For any $a\in A\setminus\{0\}$, the equation $x-y=a$ has at most $t$ solutions with $x,y\in T$. 
\item $|T|\leq \sqrt{t|A|}+\frac{1}{2}$.
\end{enumerate}
\end{lemma}
\begin{proof}
(i) $x=y+a$ if and only if $x\in T\cap(T+a)$. (ii) $|T|(|T|-1)\leq t(|A|-1)$. 
\end{proof}

The upper bound for the size of a $t$-thin Sidon set is almost sharp. Caicedo, Martos and Trujillo \cite{Caicedo2015} constructed $t$-thin Sidon sets of size $\geq \sqrt{t|A|+1}$, where $A$ is the cyclic group of order $(q^2-1)/t$, $q$ is a prime power, and $t$ a divisor of $q-1$. 

\begin{lemma} \label{lm:tSidon}
\begin{enumerate}[(i)]
\item Sidon sets are $2$-thin Sidon sets in general. 
\item If $A$ has odd order, then $S\subseteq A$ is Sidon if and only if it is a $1$-thin Sidon set.
\item If $A$ is an elementary abelian $2$-group, then $S\subseteq A$ is Sidon if and only if it is a $2$-thin Sidon set. 
\end{enumerate}
\end{lemma}
\begin{proof}
Assume that $S$ is Sidon, $a\in A\setminus\{0\}$, and $x,y\in S\cap (S+a)$ with $x\neq y$. Then 
\[a=x-x'=y-y'\]
with $x',y'\in S$. As $a\neq 0$, $x\neq x'$ and $y\neq y'$. It follows $x=y'$ and $y=x'$. In particular, $x$ and $a$ determine $y$, hence (i) holds. If $A$ has odd order, then $x=y'$ and $y=x'$ implies $2x=2y$, which contradicts $x\neq y$ in a group of odd order. Conversely, if $|A|$ is odd and $S$ is $1$-thin, then $x-x'=y-y'$ implies $x=y$ and $x'=y'$, or $x=x'$ and $y=y'$, by Lemma \ref{lm:tSidon}(i). Assume that $A$ has exponent $2$, $S$ is $2$-thin, and $x+y=z+w$ with $x,y,z,w\in S$. If $x=y$, then $z=w$. If $x\neq y$, then $\{x,y\}=\{z,w\}$ by Lemma \ref{lm:tSidon}(i).  
\end{proof}

In \eqref{eq:cubeexample}, we have seen the classical example of a Sidon set in the elementary abelian $2$-group $\mathbb{F}_q^2$, $q=2^m$. The functions $x^3$ can be replaced by any APN function $f:\mathbb{F}_q\to \mathbb{F}_q$. More precisely, we have the following \textit{folklore} result, see \cite{Carlet2022}. 
\begin{lemma}
The function $f:\mathbb{F}_q\to \mathbb{F}_q$ is APN if and only it its graph $\{(x,f(x)) \mid x\in \mathbb{F}_q\}$ is a Sidon set in $\mathbb{F}_q^2$. 
\end{lemma}
In this way, many Sidon sets of size $q=\sqrt{|A|}$ exist in $A$. Much less is known about Sidon sets of size $\sqrt{|A|}+1$. We present the best known examples in the next section.

\section{Sidon sets and double-error correcting codes} \label{sec:sidon-codes}

A \textit{linear code} with parameters $[N,K,d]_q$ is a linear subspace $C\leq \mathbb{F}_q^N$, with $\dim(C)=K$ and minimum Hamming distance $d$. A linear code is either given by its $K\times N$ \textit{generator matrix} $G$, or by its $(N-K)\times N$ \textit{parity check matrix} $H$:
\[C=\{xG \mid x \in \mathbb{F}_q^K\} = \{y\in \mathbb{F}_q^N \mid yH^T=0\}.\]
The linear code $C$ has minimum distance at least $d$ if and only if any $d-1$ columns of its parity check matrix are linearly independent. The code is said to be $t$-error correcting, if $2t+1\leq d$. Two codes $C_1,C_2\leq \mathbb{F}_q^N$ are \textit{permutation equivalent,} if there is a permutation $\pi \in S_N$ such that
\[(x_1,\ldots,x_n)\in C_1 \Longleftrightarrow (x_{1^\pi},\ldots,x_{n^\pi})\in C_2.\]
Permutation equivalent codes have the same parameters. 

Let $T$ be any subset of $\mathbb{F}_2^n$. Let $H_T$ be the $n\times (s-1)$ matrix whose columns are the elements of $T$. Let $C_T$ be the binary linear code with parity check matrix $H_T$. Strictly speaking, $H_T$ is only defined up to an ordering of the non-zero elements of $S$, and $C_T$ is defined up to permutation equivalence. Since this does not change the parameters, we use the sloppy notation $H_T$ and $C_T$. 

The relation between Sidon sets of $\mathbb{F}_2^n$ and linear codes of minimum distance $5$ seems to be a \textit{folklore} known fact, see \cite{Redman2021}. In the context of APN functions, we have to mention the seminal CCZ paper \cite{Carlet1998}, where the authors associated a linear $[N=2^m-1,k,5]_2$ code $C_F$ to the APN function $F$ by the parity check matrix
\[\begin{bmatrix}
1 & \alpha &  \alpha^2 & \cdots & \alpha^{N-1} \\
F(1) & F(\alpha) &  F(\alpha^2) & \cdots & F(\alpha^{N-1}) 
\end{bmatrix}.\] 

For the sake of completeness, we formulate the relation between Sidon sets and double-error correcting codes in the following lemma. Remember that for any Sidon set $S$, $S-a$ is Sidon too. Hence, we may assume $0\in S$ without loss of generality. Moreover, since we are looking for complete Sidon sets, we may assume that $S$ spans $\mathbb{F}_2^n$. 

\begin{lemma} \label{lm:sidon_dist5}
Let $S$ be a Sidon set of $\mathbb{F}_2^n$ with $0\in S$. Then $C_{S\setminus\{0\}}$ is a binary linear code of length $|S|-1$, dimension $|S|-1-n$, and minimum distance al least $5$. Conversely, let $C$ be a linear $[N,K,\geq 5]_2$-code with parity check matrix $H$. Let $T$ be the set of column vectors of $H$. Then $T\cup \{0\}$ is a Sidon set of size $N+1$ in $\mathbb{F}_2^{N-K}$. 
\end{lemma}
\begin{proof}
Write $T=S\setminus\{0\}$, and $d$ be the minimum distance of $C_T$. As $H_T$ has no zero column, $d>1$. The columns are different, hence $d>2$. Since $T\cup \{0\}$ is Sidon, $x_1+x_2+x_3=0$ has no solution in $T$, and $d>3$. Finally, $T$ itself is Sidon, and $x_1+x_2+x_3+x_4=0$ has no solution in $T$. This implies $d>4$. The proof of the converse is similar. 
\end{proof}

\begin{table}[]
\caption{The maximum sizes of Sidon sets in $\mathbb{F}_2^n$\label{tab:smalln}}
\begin{tabular}{l@{\hspace{1.5cm}}l@{\hspace{1.5cm}}l}
$n$ & $2^{n/2}$ & max $|S|$ \\ \hline
2   & $2$       & 3         \\
3   & $2.83$    & 4         \\
4   & $4$       & 6         \\
5   & $5.66$    & 7         \\
6   & $8$       & 9         \\
7   & $11.31$   & 12        \\
8   & $16$      & 18        \\
9   & $22.63$   & 24        \\
10  & $32$      & 34        \\
11  & $45.25$   & 48–58     \\
12  & $64$      & 66–89     \\
13  & $90.51$   & 82–125    \\
14  & $128$     & 129–179   \\
15  & $181.02$  & 152–254  
\end{tabular}
\end{table}

Double-error correcting binary linear codes with good parameters are known in the literature.

\begin{enumerate}[({C}1)]
\item The online database \cite{Grassl2007} of linear codes gives the main parameters of known linear codes up to length $256$. This allows us to compile Table \ref{tab:smalln} with the maximum sizes of Sidon sets in $\mathbb{F}_2^n$ for $n\leq 15$. 
\item Shortening of non-primitive binary BCH codes gives $[2^m,2^m-2m,5]_2$-codes, see \cite[p. 586]{MacWilliams1977}. These give rise to Sidon sets of size $2^{n/2}+1$ in $\mathbb{F}_2^n$ for $n=2m$. 
\item For $n=2m+1$ odd, shortened BCH $[2^m+2^{\lfloor (m+1)/2 \rfloor}-1,2^m+2^{\lfloor (m+1)/2 \rfloor}-2m-2, 5]$-code exists. The magnitude of the size of the corresponding Sidon set $S$ is $\frac{1}{\sqrt{2}} 2^{n/2}$. For small odd integers $n$, we have
\[\begin{array}{c|cccc}
n=2m+1 & 9 & 11 & 13 & 15 \\ \hline 
|S|=2^m+2^{\lfloor (m+1)/2 \rfloor} & 20 & 40 & 72 & 144
\end{array}\]
Comparison with Table \ref{tab:smalln} shows that for small $n$, these codes and Sidon sets are rather far from being optimal. 
\item Let $q=2^m$, $g(X)\in \mathbb{F}_q[X]$ irreducible of degree $2$, and $L=(x_1,\ldots,x_q)$ with $\{x_i\}=\mathbb{F}_q$. The binary Goppa code $\Gamma(L,g)$ has minimum distance $\geq 5$, length $2^m$, and dimension $\leq 2^m-2m$. In fact, the dimension equals $2^m-2m$ by the main result of \cite{Vlugt1990}. The corresponding Sidon set has size $q+1$.
\end{enumerate}

\begin{remark} \label{rem:goppa}
The generic element of the parity-check matrix of the binary Goppa code $\Gamma(L,g)$ has the form $t^i/g(t)$, $t\in L$. The Sidon set (C4) consists of the origin $(0,0)$ and the points
\[\left(\frac{1}{g(t)}, \frac{t}{g(t)}\right), \qquad t\in \mathbb{F}_q.\]
This is the set of $\mathbb{F}_q$-rational points of the conic $x^2 g(y/x)+x=0$.
\end{remark}

\section{Nonlinearity of APN functions and Sidon sets} \label{sec:nonlin}

Although not explicitely stated, Czerwinski has observed in \cite[Proposition 2.1]{Czerwinski2020}, that the level sets of APN functions are Sidon sets in $\mathbb{F}_2^n$. The next lemma makes this statement more precise by linking the $t$-thin parameter to the differential uniformity.

\begin{lemma} \label{lm:deltafthin}
The level sets of the map $f:V\to V'$ are $\delta_f$-thin Sidon sets. 
\end{lemma}
\begin{proof}
Write $S=f^{-1}(b)$ for $b\in V'$. For any $a\in V\setminus\{0\}$, we have
\begin{align*}
x \in S\cap (S+a) &\Leftrightarrow f(x)=b \text{ and } f(x+a)=b \\
&\Leftrightarrow f(x)=b \text{ and } f(x)+f(x+a)=0.
\end{align*}
As $f(x)+f(x+a)=0$ has at most $\delta_f$ solutions, $|S\cap (S+a)|\leq \delta_f$. 
\end{proof}

The proof of the next lemma is analogous to \cite[Proposition 2]{Carlet1998}, the details are left to the reader. 

\begin{lemma} \label{lm:samedelta}
Let $f,\alpha:V\to V'$ be maps. If $\alpha$ is affine then $f$ and $f+\alpha$ have the same differential uniformity.
\end{lemma}

\begin{proof}[Proof of Theorem \ref{thm:nonlinearity}]
Let $f,\alpha:V\to V'$ be maps and assume that $\alpha$ is affine. Let $S$ be the set of $x\in V$ such that $f(x)=\alpha(x)$. On the one hand, 
\[d_H(f,\alpha)=|V|-|S|=2^n-|S|.\]
On the other hand, $S=(f+\alpha)^{-1}(0)$, and it is $\delta_f$-thin Sidon by Lemma \ref{lm:deltafthin} and \ref{lm:samedelta}. This implies
\[|S|\leq \sqrt{\delta_f \cdot 2^n} + \frac{1}{2},\]
and
\[d_H(f,\alpha)\geq 2^n-\sqrt{\delta_f} \cdot 2^{n/2} - \frac{1}{2} \qedhere\]
\end{proof}

For APN functions, Theorem \ref{thm:nonlinearity} uses the upper bound $|S|\leq \sqrt{2} \cdot 2^{n/2}$ for Sidon sets in $\mathbb{F}_2^n$. With regard to the problem on Sidon sets in elementary abelian 2-groups, we mentioned that all but one known Sidon sets of $\mathbb{F}_2^n$ have size $\leq 2^{n/2}+2$. This motivates the following proposition.

\begin{proposition} \label{pr:APNstrong}
Let $n$ be a positive integer such that $|S|\leq 2^{n/2}+2$ holds for all Sidon sets of $\mathbb{F}_2^n$. Then for any APN function $f:\mathbb{F}_2^n\to \mathbb{F}_2^n$ we have
\[\NL(f)\geq 2^n-2^{n/2}-2.\]
\end{proposition}

Notice that the lower bound of Proposition \ref{pr:APNstrong} is very close to the upper bound of the Liu-Mesnager-Chen Conjecture.

\section{Conics as Sidon sets} \label{sec:conics}

In this section, $n=2m$ and $q=2^m$. The underlying abelian group is $A=\mathbb{F}_2^{2m} \cong \mathbb{F}_q\times \mathbb{F}_q \cong \mathbb{F}_{q^2}$. Recently, Carlet and Mesnager \cite{Carlet2022} showed the existence of Sidon sets of size $q+1$ in $A$. They proved that the cyclic multiplicative subgroup of order $q+1$ is a Sidon set in $\mathbb{F}_{q^2}$. In this section, we identify this set as an ellipse of the affine plane $AG(2,q)$. We give an independent proof for the Sidon property, that will also show that for even $m$, adding $0$ to the subgroup still has the Sidon property. Similar results hold for hyperbolas of $AG(2,q)$, as well. 

Let $\gamma$ be an element of $\mathbb{F}_q$ such that the roots $\delta, 1/\delta$ of the polynomial $X^2+\gamma X+1$ are in $\mathbb{F}_{q^2}\setminus\mathbb{F}_q$. We use the map $\Delta:(x,y) \mapsto x+\delta y$ to identify $\mathbb{F}_q^2$ and $\mathbb{F}_{q^2}$. Any affine collineation is the composition of a linear map, a translation, and a Frobenius map $(x,y)\mapsto (x^{2^k},y^{2^k})$. By $\Delta$, an affine collineation induces a group automorphism of $A$. The converse is not true, the map $(x,y)\mapsto (x,y^2)$ is a counter-example. 

A \textit{conic} $C$ of $AG(2,q)$ is given by an irreducible quadratic equation $Q(X,Y)=0$. Using the composition of an $\mathbb{F}_q$-linear map and a translation, the equation can be transformed to one of the following forms:
\begin{align*}
\text{hyperbola:} \qquad & H:XY=1, \\
\text{parabola:} \qquad & P:Y=X^2, \\
\text{ellipse:} \qquad & E:X^2+\gamma XY+Y^2=1. 
\end{align*}
The number of $\mathbb{F}_q$-rational points of a hyperbola, parabola or ellipse is $q-1$, $q$, or $q+1$, respectively. The latter follows from the parametrization
\begin{align}
\left\{\left(\frac{\gamma}{t^2+\gamma t+1},\frac{t^2+1}{t^2+\gamma t+1}\right) \mid t\in \mathbb{F}_q\right\} \cup \{(0,1)\}
\end{align}
of the ellipse $E$. This implies that ellipses, hyperbolas and parabolas cannot be equivalent neither under affine maps, nor under additive automorphisms. 

Another important fact is that the tangents of a conic pass through a common point. This point is called the \textit{nucleus} of the conic. For $H$ and $E$, the nucleus is the origin, the nucleus of the parabola $P$ is the point at infinity of the $y$-axis. 

When we switch to the projective coordinate frame $(x_1:x_2:x_3)$, $x=x_1/x_3$, $y=x_2/x_3$, we see that the hyperbola $H$ has two points at infinity $(1:0:0)$ and $(0:1:0)$, the parabola $P$ has one point at infinity $(0:1:0)$, and the points at infinity $(1:\delta:0)$, $(\delta:1:0)$ of the ellipse $E$ are defined over $\mathbb{F}_{q^2}$. 

\begin{lemma} \label{lm:groups}
Let $C:Q(X,Y)=0$ be an ellipse or a hyperbola of $AG(2,q)$, $q=2^m$. Let $N$ denote the nucleus of $C$. For $s\in \mathbb{F}_q$, define the set 
\[\mathcal{D}_s=\{(x,y) \mid Q(x,y)=s\}\]
of affine points. 
\begin{enumerate}[(i)]
\item There is a cyclic affine linear group $G$ that preserves $N$ and acts regularly on $C$. The only fixed point of $G$ is $N$. 
\item If $s\neq 0$, then $\mathcal{D}_s$ is a $G$-orbit. 
\item If $C=E$, then $\mathcal{D}_0=\{N \}$. 
\item If $C=H$, then $\mathcal{D}_0$ is the union of the two asymptotes, intersecting in $N$. $\mathcal{D}_0$ decomposes into one orbit of size $1$ (the nucleus), and two orbits of size $q-1$ (the punctured asymptotes). 
\end{enumerate} 
\end{lemma}
\begin{proof}
We choose affine coordinates such that $C$ is either $E$ or $H$. The hyperbola $H$ is left invariant by the cyclic linear group 
\[G_H=\left\{\begin{bmatrix} a&0 \\ 0&a^{-1}\end{bmatrix} \mid a \in \mathbb{F}_q^*\right\}.\]
The ellipse $E$ is left invariant by the cyclic linear group
\[G_E=\left\{\begin{bmatrix} a&b \\ b&a+\gamma b\end{bmatrix} \mid (a,b)\in E\right\}.\]
The statements on the orbits follow by direct calculation.
\end{proof}

\begin{proposition} \label{pr:conicsidon}
Let $C$ be an ellipse or a hyperbola in $AG(2,q)$, $q=2^m$. 
\begin{enumerate}[(i)]
\item For all $q$, $C$ is a Sidon set in $\mathbb{F}_q^2$.
\item Let $N$ be the nucleus of $C$. $C\cup \{N\}$ is Sidon if and only if $3$ does not divide $|C|$. 
\end{enumerate}
\end{proposition}
\begin{proof}
Fix an element $a\in \mathbb{F}_q^2$, $a\neq (0,0)$, and let $\tau$ be the translation $x\mapsto x+a$. As $C$ has an odd number of points, and $\tau$ is fixed point free, we cannot have $C=\tau(C)$. The two conics $C$ and $\tau(C)$ have two points in common on the line at infinity. Hence, they cannot have more than 2 affine points in common. This shows $|C\cap (C+a)|\leq 2$, which implies (i) by Lemma \ref{lm:tSidon}(iii).

Let us now assume that $C\cup \{N\}$ is not Sidon, that is, 
\[P_1+P_2+P_3=N\] 
with $P_1,P_2,P_3 \in C$. There is a unique affine transformation $\alpha$ which maps $P_1,P_2,P_3$ to $P_2,P_3,P_1$, respectively. $\alpha$ has order $3$ and $\alpha(N)=N$. The nucleus of the conic $C'=\alpha(C)$ is $N=\alpha(N)$. This implies that $C$ and $C'$ share the points $P_1,P_2,P_3$ and the tangents $NP_1$, $NP_2$, $NP_3$. It follows that $C=C'$. The total number of points of $AG(2,q)$ is $q^2\equiv 1 \pmod{3}$, and $\alpha$ cannot have more than $2$ fixed points. Therefore the only fixed point of $\alpha$ is $N$, and $3$ divides $|C|$. 

Conversely, assume that $3$ divides $|C|$. Then there is an affine transformation $\alpha$ of order $3$ such that $C=\alpha(C)$. The nucleus $N$ is fixed by $\alpha$, and $\alpha$ has no further fixed points. For any $P\in C$, the points $P$, $\alpha(P)$, $\alpha^2(P)$ are distinct. As their sum is fixed by $\alpha$, we must have
\[P+\alpha(P)+\alpha^2(P)=N.\]
This shows that $C\cup \{N\}$ is not Sidon and the proof of (ii) is complete. 
\end{proof}

\begin{remark} \label{rem:hyperbola}
In Proposition \ref{pr:conicsidon}, the case when $C$ is a hyperbola and $3$ does not divide $|C|$ is not new. As $|C|=q-1$, the divisibility condition is equivalent with $m$ being odd. Choose the coordinate frame such that $C:XY=1$. Then $N=(0,0)$ and $C\cup\{N\}$ is the graph of the function $f(x)=x^{q-2}$. This function, used as S-box in AES, is well-known to be APN, hence its graph is Sidon. 
\end{remark}

\begin{remark} \label{rem:ellipse}
\begin{enumerate}[a)]
\item The bijection $\Delta:\mathbb{F}_{q}^2\to \mathbb{F}_{q^2}$ identifies the Carlet-Mesnager Sidon set of $\mathbb{F}_{q^2}$ with the points of the ellipse in $\mathbb{F}_q^2$. The above result shows that the Carlet-Mesnager set can be extended to a Sidon set of size $q+2$ for $m$ even. 
\item By Remark \ref{rem:goppa}, the Sidon set (C4) obtained from the binary Goppa code is also equivalent with the ellipse $E$. 
\item Using the terminology of double-error correcting codes, Chen \cite{Chen1991} noticed that for $m\in \{4,5,6\}$, (C4) can be extended by one point. We will see later that the extension is unique, hence Chen's result is a particular case of Proposition \ref{pr:conicsidon}. 
\end{enumerate}
\end{remark}

\section{Completness results} \label{sec:completeness}

We continue using the notation of the previous section, $m$ is a positive integer, $q=2^m$, $\gamma \in \mathbb{F}_q$ such that $X^2+\gamma X+1$ is irreducible over $\mathbb{F}_q$, and $\delta,1/\delta \in \mathbb{F}_{q^2}$ are the roots of $X^2+\gamma X+1$.

\begin{lemma} \label{lm:hyperbolacomplete}
Let $q\geq 16$ be a power of two, and $c\in \mathbb{F}_q\setminus\{0,1\}$. There are elements $x,y\in \mathbb{F}_q\setminus\{0\}$ such that
\[x^2y+xy^2+c xy+x^2+y^2+x+y=0\]
\end{lemma}
\begin{proof}
The algebraic plane curve $\Gamma_c:X^2Y+XY^2+cXY+X^2+Y^2+X+Y=0$ has degree $3$ with three distinct points at infinity. The line at infinity is not a component, hence, the points at infinity of $\Gamma_c$ are smooth. Assume that $(x,y)$ is a singular affine point of $\Gamma_c$. Taking partial derivatives, we obtain $(x+y)(x+y+c)=0$. If $x=y$, then $cxy=0$, which implies $x=y=0$. However, $(0,0)$ is a smooth point. Furthermore, plugging $y=x+c$ into the equation of $\Gamma_c$, we get $c^2+c\neq 0$. This shows that $\Gamma_c$ has no singular points, in particular, $\Gamma_c$ is irreducible of genus $1$. By the Hasse-Weil Bound \cite[Theorem 9.18]{Hirschfeld2008}, the number of $\mathbb{F}_q$-rational point of $\Gamma_c$ is at least $q+1-2\sqrt{q}$. Not counting the three points at infinity, $(0,0)$, $(1,0)$ and $(0,1)$, we obtain at least
\[q-5-2\sqrt{q}>0\]
solutions for $(x,y)$ if $q\geq 16$. 
\end{proof}

\begin{lemma} \label{lm:ellipsecomplete}
Let $q\geq 16$ be a power of two, and $c$ an element of $\mathbb{F}_q$, with $c\neq 0, \gamma^3$. Define the homogeneous polynomial
\[F_c(X,Y,Z) = (X^2+\gamma XZ+Z^2)(Y^2 + \gamma YZ+Z^2) + c (X+Y)Z^3.\]
\begin{enumerate}[(i)]
\item The projective plane algebraic curve $\Gamma_c:F_c=0$ has two ordinary singularities at $P_1(1:0:0)$ and $P_2(0:1:0)$. All other points (over the algebraic closure of $\mathbb{F}_q$) are nonsingular.
\item $\Gamma_c$ is absolutely irreducible. 
\item The genus of $\Gamma_c$ is $1$.
\item There are elements $x,y\in \mathbb{F}_q$, $x\neq y$, such that $F_c(x,y,1)=0$. 
\end{enumerate}
\end{lemma}
\begin{proof}
(i) Over $\mathbb{F}_{q^2}$, we have
\begin{align*}
F_c(X,Y,Z) = (X+\delta Z)(X+1/\delta Z)(Y+\delta Z)(Y+1/\delta Z) + c (X+Y) Z^3,
\end{align*}
This shows that $P_1,P_2$ are indeed ordinary singularities with tangent lines
\begin{align} \label{eq:tangents}
m_1:X=\delta Z, \quad \ell_1:X=1/\delta Z, \quad m_2:Y=\delta Z, \quad \ell_2:Y=1/\delta Z.
\end{align}
$\Gamma_c$ has no further points on the line at infinity. Let $P(x:y:1)$ denote a singular point. Taking partial derivatives with respect to $X,Y$, we obtain
\[x^2+\gamma x+1=y^2+\gamma y+1 = c/\gamma.\]
On the one hand, $F_c(x,y,1)=0$ implies $x+y=c/\gamma^2$. On the other hand, $x^2+\gamma x+1=y^2+\gamma y+1$ implies 
\[(x+y)(x+y+\gamma)=0.\]
Hence,
\[0=\frac{c}{\gamma^2}\left(\frac{c}{\gamma^2}+\gamma\right)=\frac{c(c+\gamma^3)}{\gamma^2},\]
which contradicts to the choice of $c$. This proves (i).

(ii) Assume first that $\Gamma_c$ decomposes into the product of two (possibly reducible) quadratic curves: $F_c=Q_1Q_2$. Write $C_i:Q_i=0$, $i=1,2$. If $C_1$ does not pass through $P_1$, then $P_1$ is a singular point of $C_2$, hence $C_2$ is reducible. This implies $Q_2(X,Y,Z)=(Y+\delta Z)(Y+1/\delta Z)$, and $Q_1(X,Y,Z)=(X+\delta Z)(X+1/\delta Z)$ as $F_c$ is symmetric in $X,Y$. Hence, $\Gamma_c$ is the union of the four tangent lines, which is not possible due to the nonzero terms $\gamma^3(X+Y)$. This shows that $P_1,P_2$ are smooth points of $C_1$ and $C_2$. Moreover, the lines $m_1,m_2,\ell_1,\ell_2$ given in \eqref{eq:tangents} are tangents of $C_1, C_2$. We can assume w.l.o.g. that $X=\delta Z$ is the tangent of $C_1$ at $P_2$. At $P_1$, the tangent of $C_1$ is either (Case 1) $Y=1/\delta Z$ or (Case 2) $Y=\delta Z$. 

Case 1: We have
\begin{align*}
Q_1(X,Y,Z)&=(X+\delta Z)(Y+1/\delta Z) + a_1 Z^2,\\
Q_2(X,Y,Z)&=(X+1/\delta Z)(Y+\delta Z) + a_2 Z^2.
\end{align*}
Since $F_c$ is symmetric in $X,Y$, we must have $Q_1(X,Y,Z)=Q_2(Y,X,Z)$, and $a_1=a_2$. The points $D_1(\delta:\delta:1)$, $D_2(\frac{1}{\delta}:\frac{1}{\delta}:1)$ are on $\Gamma_c$. Plugging them into $Q_1,Q_2$, we get $a_1=a_2=0$, a contradiction. 

Case 2: We have
\begin{align*}
Q_1(X,Y,Z)&=(X+\delta Z)(Y+\delta Z) + a_1 Z^2,\\
Q_2(X,Y,Z)&=(X+1/\delta Z)(Y+1/\delta Z) + a_2 Z^2.
\end{align*}
As $F_c$ is defined over $\mathbb{F}_q$, $a_1^q=a_2$ holds. Moreover, $Q_1(\delta,\delta,1)=a_1\neq 0$, hence $0=Q_2(\delta,\delta,1)=\gamma^2+a_2$. This implies $a_1=a_2=\gamma^2$, which contradicts to $F_c=Q_1Q_2$. 

For (ii), it remains to show that $\Gamma_c$ cannot decompose into the product of a linear and an irreducible cubic factor. If this would be the case, since $F_c$ is symmetric in $X,Y$, the linear component had the shape $\ell:X+Y=aZ$. Then, $\Gamma_c$ had $(1:1:0)$ as a point at infinity, a contradiction. 

(iii) follows from the genus formula, see \cite[Theorem 5.57]{Hirschfeld2008}. (iv) By the Hasse-Weil Bound \cite[Theorem 9.18]{Hirschfeld2008}, the number of $\mathbb{F}_q$-rational places of $\Gamma_c$ is at least $q-2\sqrt{q}+1$. The two points at infinity correspond to 4 branches of order 1, and each affine point correspond to a unique branch of order 1, we have at least $q-3-2\sqrt{q}>0$ affine points of $\Gamma_c$ over $\mathbb{F}_q$. No such point can lay on the line $X=Y$. This finishes the proof.
\end{proof}

\begin{proposition} \label{pr:completeness}
Let $C$ be an ellipse or a hyperbola in $\mathbb{F}_q^2$, $q=2^m\geq 16$. Let $N$ be the nucleus of $C$ and $P$ a point of $AG(2,q)$ not contained in $C\cup \{N\}$. Then $C\cup\{P\}$ is not Sidon in $\mathbb{F}_q^2$.  
\end{proposition}
\begin{proof}
Let us choose the affine coordinate frame such that $C$ is either the hyperbola $H:XY=1$, or the ellipse $E:X^2+\gamma XY+Y^2=1$. In both cases, $C$ is given by $Q(X,Y)=1$, where $Q(X,Y)$ is a nonsingular quadratic form over $\mathbb{F}_q$. Let $G$ be the cyclic affine linear group of Lemma \ref{lm:groups}, with orbits $\mathcal{D}_s$, $s\in \mathbb{F}_q$, and $\mathcal{D}^{(1)}_0$, $\mathcal{D}^{(2)}_0$. We have
\[\mathcal{D}_0=\{(0,0)\} \cup \mathcal{D}^{(1)}_0 \cup \mathcal{D}^{(2)}_0, \]
where
\begin{align*}
\mathcal{D}^{(1)}_0 &= \{(0,y) \mid y\in \mathbb{F}_q\setminus \{0\}\} ,\\
\mathcal{D}^{(2)}_0 &= \{(x,0) \mid x\in \mathbb{F}_q\setminus \{0\}\}
\end{align*}
are $G$-orbits. 

Let $\mathcal{T}$ be the set of points that are the sum of three distinct elements of $C$. We show that $P\in \mathcal{T}$, this implies the proposition. Since $\mathcal{T}$ is $G$-invariant,
it is the union of $G$-orbits. Moreover, $C=\mathcal{D}_1$, $N=(0,0)$ implies that $P\in \mathcal{T}$ is equivalent with the fact that all $G$-orbits, different from $\{(0,0)\}$ or $\mathcal{D}_1$, have nonempty intersection with $\mathcal{T}$. 

We first consider the case $C=H$, $\mathcal{O}=\mathcal{D}_s$ with $s\not\in \{0,1\}$. The sum of the points $(x,\frac{1}{x})$, $(y,\frac{1}{y})$, $(1,1)$ of $H$ is in $\mathcal{O}$ if and only if
\[Q\left(x+y+1,\frac{1}{x}+\frac{1}{y}+1\right)=(x+y+1)\left(\frac{1}{x}+\frac{1}{y}+1\right)=s.\]
Equivalently,
\[x^2y+xy^2+(s+1) xy+x^2+y^2+x+y=0\]
holds with $x,y\in \mathbb{F}_q\setminus\{0\}$. The existence of such $x,y$ follows from Lemma \ref{lm:hyperbolacomplete} with $c=s+1$. 

The next case is $C=H$, $\mathcal{O}=\mathcal{D}_0^{(1)}$. Let $x\in\mathbb{F}_q\setminus\mathbb{F}_4$ be arbitrary, and let $y=x+1$. We have
\[\frac{1}{x}+\frac{1}{y}+1=\frac{x^2+x+1}{x^2+x}\neq 0,\]
which implies $(x,\frac{1}{x})+(y,\frac{1}{y})+(1,1) \in \mathcal{D}_0^{(1)}$. Similarly, $\mathcal{D}_0^{(2)}\cap \mathcal{T} \neq \emptyset$. The claim therefore holds for $C=H$. 

For the rest of the proof, we assume $C=E$, $s\in \mathbb{F}\setminus\{0,1\}$. We use the parametrization \eqref{eq:s1} of $E$. The sum of the points
\begin{align} \label{eq:3pts}
\left(\frac{\gamma}{Q(x,1)},\frac{x^2+1}{Q(x,1)}\right), \quad \left(\frac{\gamma}{Q(y,1)},\frac{y^2+1}{Q(y,1)}\right), \quad (0,1)
\end{align}
of $E$ is in $\mathcal{D}_s$ if and only if
\begin{align} \label{eq:s1}
Q\left(
	\frac{\gamma}{Q(x,1)}+\frac{\gamma}{Q(y,1)},
	\frac{x^2+1}{Q(x,1)} + \frac{y^2+1}{Q(y,1)}+1 
\right) = s
\end{align}
holds for distinct values $x,y \in \mathbb{F}_q$. We have
\begin{align*}
A&=\frac{\gamma}{Q(x,1)}+\frac{\gamma}{Q(y,1)} &=\quad& \frac{\gamma Q(x,1)+ \gamma Q(y,1)}{Q(x,1)Q(y,1)},\\
B&=\frac{x^2+1}{Q(x,1)} + \frac{y^2+1}{Q(y,1)}+1 &=\quad& \frac{(y^2+1) Q(x,1)+ \gamma x Q(y,1)}{Q(x,1)Q(y,1)}.
\end{align*}
Furthermore, 
\begin{multline*}
Q\big(\gamma Q(x,1)+ \gamma Q(y,1),(y^2+1) Q(x,1)+ \gamma x Q(y,1)\big) = \\
Q(x,1)Q(y,1)\;\big( Q(x,1)Q(y,1) +\gamma^3(x+y)\big),
\end{multline*}
thus we obtain
\[Q(A,B)=\frac{Q(x,1)Q(y,1) +\gamma^3(x+y)}{Q(x,1)Q(y,1)}.\]
Since $Q(x,1),Q(y,1)\neq 0$, $Q(A,B)=s$ is equivalent with 
\begin{align*}
(s+1)Q(x,1)Q(y,1) +\gamma^3(x+y)=0.
\end{align*}
By the definition of $Q(X,Y)$, this is precisely
\begin{align} \label{eq:normeq} 
(s+1)(x^2+\gamma x+1)(y^2+\gamma y +1) +\gamma^3(x+y)=0.
\end{align}
As $s\neq 0,1$, we may apply Lemma \ref{lm:ellipsecomplete}(iv) with $c=\gamma^3/(s+1)$. We conclude that there are $x,y\in \mathbb{F}_q$, $x\neq y$ such that \eqref{eq:normeq} holds. With these values of $x,y$, the sum of the points \eqref{eq:3pts} of $E$ are in $\mathcal{D}_s$. This finishes the proof. 
\end{proof}

\section{Proof of Theorem \ref{thm:conics}} \label{sec:proof}

\begin{proof}[Proof of Theorem \ref{thm:conics}]
$|C|$ is divisible by $3$ if and only if $m$ is even and $C$ is a hyperbola, or $m$ is odd and $C$ is an ellipse. Hence, Propositions \ref{pr:conicsidon} and \ref{pr:completeness} imply the theorem. 
\end{proof}

\begin{remark}
\begin{enumerate}[a)]
\item Let $C$ be an ellipse. If $m\leq 3$, then the size of $|C|$ (if $m=2$) or $|C\cup\{N\}|$ (if $m=1,3$) is $3,6$ or $9$. By Table \ref{tab:smalln}, these are Sidon sets of maximal size, which are of course complete. 
\item Let $C$ be a hyperbola. Straightforward calculation shows that neither $C$ (if $m=2$), nor $C\cup\{N\}$ (if $m=1$) are complete as Sidon sets. If $m=3$, then $C\cup\{N\}$ is complete.
\item If $C$ is a hyperbola and $m$ is odd, then the completeness of $C\cup \{N\}$ was shown in \cite[Section 6]{Carlet2022a} by Carlet, see also Remark \ref{rem:hyperbola}. 
\end{enumerate}
\end{remark}

\noindent {\bf \large Acknowledgements}. I thank Claude Carlet for many valuable remarks, references on APN functions, and results concerning the completeness of their graph. I thank L\'aszl\'o Babai, J\'ozsef Balogh and Lajos R\'onyai for useful information on Sidon sets in elementary abelian groups. I am also grateful to the organizers of the eighth international Olympiad in cryptogrphy NSUCRYPTO \cite{Gorodilova2022}, because this research has been motivated by Problem 11 ``Distance to affine functions'' of the Olympiad.

\printbibliography

\end{document}